\documentclass[12pt]{article}
\usepackage{amsmath,amssymb,amsthm}

\theoremstyle{plain}
\newtheorem{theorem}{Theorem}[section]

\newtheorem{lemma}[theorem]{Lemma}
\newtheorem{corollary}[theorem]{Corollary}
\theoremstyle{definition}

\newcommand{\cald}{{\cal D}}
\newcommand{\calg}{{\cal G}}
\newcommand{\calk}{{\cal K}}

\def\0{\leqno}

\title{On finite groups with dismantlable\\ subgroup lattices}
\author{Marius T\u arn\u auceanu}
\date{February 17, 2015}

\begin{document}
\maketitle

\begin{abstract}
    In this note we study the finite groups whose subgroup
    lattices are dismantlable.
\end{abstract}

{\small \noindent {\bf MSC (2000)\,:} Primary 20D30; Secondary
20D60, 20E15.

\noindent {\bf Key words\,:} finite groups, subgroup lattices,
dismantlable lattices, planar lattices, crowns.}

\section{Introduction}

The relation between the structure of a group and the structure of
its lattice of subgroups constitutes an important domain of
research in group theory. The topic has enjoyed a rapid
development starting with the first half of the 20th century. Many
classes of groups determined by different properties of partially
ordered subsets of their subgroups (especially lattices of
(normal) subgroups) have been identified. We refer to Suzuki's
book \cite{9}, Schmidt's book \cite{7} or the more recent book
\cite{11} by the author for more information about this theory.
\smallskip

A finite lattice $L$ of $n$ elements is called
\textit{dismantlable} if there is a chain $L_1\subset
L_2\subset...\subset L_n=L$ of sublattices of $L$ such that
$|L_i|=i$, for all $i=1,2,...,n$ (see \cite{5,6}). It is
well-known that every lattice with at most seven elements is
dismantlable, while for every integer $n\geq 8$ there is a lattice
of $n$ elements which is not dismantlable. Two basic properties of
these lattices are the following:
\begin{itemize}
\item[-] \textit{The class of dismantlable lattices is closed under the formation of sublattices and homomorphic images}.
\item[-] \textit{If a dismantlable lattice is not a chain, then it contains at least two incomparable doubly irreducible elements}.
\end{itemize}

Several characterizations of dismantlable lattices are known. One
of the most significant is given by \cite{5} and uses some
particular partially ordered sets, namely crowns. Recall that for
an integer $n\geq 3$, a \textit{crown} of order $2n$ is a poset
$\left\{x_1,y_1,x_2,y_2,...,x_n,y_n\right\}$ in which $x_i\leq
y_i$ for all $i=1,2,...,n$, $y_i\geq x_{i+1}$ for all
$i=1,2,...,n-1$, and $x_1\leq y_n$ are the only comparability
relations.

\begin{theorem}
    A finite lattice is dismantlable if and only if it contains no
    crown. In particular, a finite modular lattice is dismantlable
    if and only if it contains no crown of order\, {\rm 6}, or,
    equivalently, no sublattice isomorphic to the boolean lattice ${\bf 2}^3$.
\end{theorem}

The connections between dismantlable lattices and planar lattices
(i.e. lattices having a planar diagram) are powerful (see
\cite{1,5}). In this way, every finite planar lattice is
dismantlable. The converse is not necessary true, more precisely
for every integer $n\geq 9$ there is an $n$-element dismantlable
lattice which is not planar. However, for distributive lattices
the above two concepts are equivalent.

\begin{theorem}
    A finite distributive lattice is dismantlable if and only if it is planar.
\end{theorem}

Starting with the lattice-theoretical concept of "planar lattice",
in \cite{8} the classification of finite groups whose subgroup
lattices are planar has been made, namely:

\begin{theorem}
    A finite group $G$ has planar subgroup lattice if and only if
    it satisfies one of the following properties {\rm(}where $p$ and $q$ are primes,
    $n\in\mathbb{N}_0$ and $m\in\mathbb{N}${\rm):}
\begin{itemize}
\item[{\rm a)}] $G$ is cyclic of order $p^n$ or $p^nq^m$.
\item[{\rm b)}] $G$ is a $p$-group lattice-isomorphic to $\mathbb{Z}_{p^m}\times\mathbb{Z}_p$\,.
\item[{\rm c)}] $G$ is dihedral of order\, {\rm 8} or quaternion of order\, {\rm 8} or\, {\rm 16}.
\item[{\rm d)}] $G=PQ$, where $P\unlhd G$, $|P|=p$, $Q$ is cyclic of order $q^m$, and \newline $|Q:C_Q(P)|=q$.
\item[{\rm e)}] $G=PQ$, where $P\unlhd G$, $|P|=p^2$, $|Q|=q$, and $Q$ operates irreducibly on $P$.
\end{itemize}
\end{theorem}

Inspired by the above classification and because dismantlable
lattices and planar lattices are closely connected, the following
question is very natural:
\begin{itemize}
\item[] \textit{Which are the finite groups $G$ with dismantlable subgroup lattices}?
\end{itemize}This suggests us to consider the class $\cald$ of groups satisfying the above pro\-per\-ty.
Obviously, finite groups with at most seven subgroups are
contained in $\cald$, but at first sight it is difficult to
determine all finite groups in $\cald$. Therefore they must be
investigated more carefully. Their study is the main goal of the
current note.
\bigskip

Most of our notation is standard and will usually not be repeated
here. Basic definitions and results on lattices and groups can be
found in \cite{2,3} and \cite{4,10}, respectively.

\section{Main results}

We start with the following easy but important lemma.

\begin{lemma}
    Let $G$ be a finite group in $\cald$. Then all elements of $G$
    are of order $p^n$ or $p^n q^m$, where $p,q$ are primes and
    $n,m\in \mathbb{N}$.
\end{lemma}

\begin{proof}
Let $a\in G$. Then $L(\langle a\rangle)$ is a sublattice of $L(G)$
and consequently it is dismantlable. Since $L(\langle a\rangle)$
is also distributive (see e.g. Corollary 1.2.4 of \cite{7}), we
infer that it is planar by Theorem 1.2. Moreover, Theorem 1.3
shows that $o(a)=|\langle a\rangle|$ is of type $p^n$ or $p^n
q^m$, as desired.
\end{proof}

It is clear that the groups $A_n$ for $n\leq 4$ and $S_n$ for
$n\leq 3$ belong to $\cald$. On the other hand, we can easily
check that the subgroups $H_1=\langle (12)(34)\rangle$,
$K_1=\langle (125),(12)(34)\rangle$, $H_2=\langle
(25)(34)\rangle$, $K_2=\langle (143),(25)(34)\rangle$,
$H_3=\langle (13)(25)\rangle$ and $K_3=\langle
(15234),(13)(25)\rangle$ form a crown of order 6 in $A_5$, while
the subgroups $H_1=\langle (12)\rangle$, $K_1=\langle
(123),(12)\rangle$, $H_2=\langle (13)\rangle$, $K_2=\langle
(134),(13)\rangle$, $H_3=\langle (14)\rangle$ and $K_3=\langle
(124),(14)\rangle$ form a crown of order 6 in $S_4$. Consequently,
$A_5$ and $S_4$ are not contained in $\cald$. These remarks lead
to the following result.

\begin{theorem}
     The alternating group $A_n$ is contained in $\cald$ if and
     only if $n\leq 4$, while the symmetric group $S_n$ is contained
     in $\cald$ if and only if $n\leq 3$.
\end{theorem}

Next we will focus on describing abelian groups $G$ contained in
$\cald$. Since the subgroup lattice of an abelian group is
modular, Theorem 1.1 implies that $G$ belongs to $\cald$ if and
only if $L(G)$ does not contain sublattices isomorphic to the
boolean lattice ${\bf 2}^3$. We remark that for abelian $p$-groups
this condition is equivalent with the fact that $G$ has no section
of type $(p,p,p)$, i.e. it is of rank $\leq 2$.

\begin{theorem}
    A finite abelian group $G$ belongs to $\cald$ if and only if
    either it is cyclic and $|G|\in\{p^n,p^n q^m\}$, where $p,q$
    are primes and $n,m\in \mathbb{N}$, or $G$ is a rank\, {\rm 2}
    abelian $p$-group.
\end{theorem}

\begin{proof}
The subgroup lattice of a cyclic group of order $p^n$ or $p^n q^m$
is planar by Theorem 1.3 and therefore such a group is contained
in $\cald$. On the other hand, we already have seen that rank 2
abelian $p$-groups belong to $\cald$.
\smallskip

Conversely, assume that $G$ is contained in $\cald$. Then Lemma
2.1 and the above remark show that either $G$ is an abelian
$p$-group of rank $\leq 2$ or a direct product of two $p$-groups,
say $G\cong G_1\times G_2$ with $|G_1|=p^n$ and $|G_2|=q^m$,
$n,m\in \mathbb{N}^*$. In the second case if both $G_1$ and $G_2$
are cyclic, so is $G$ and we are done. Assume now that $G_1$ is
not cyclic and take two maximal subgroups $M$ and $M'$ of $G_1$.
Then it is easy to see that the subgroups $H_1=M$, $K_1=G_1$,
$H_2=M'$, $K_2=M'\times G_2$, $H_3=(M\cap M')\times G_2$ and
$K_3=M\times G_2$ form a crown of order 6 of $L(G)$, a
contradiction. This completes the proof.
\end{proof}

An immediate consequence of Theorems 1.3 and 2.3 is the following.

\begin{corollary}
    For every integer $n=p^m$, where $p$ is a prime and $m\geq 4$,
    there is a not-planar dismantlable lattice of $n$ elements which
    is isomorphic to the subgroup lattice of a certain finite abelian
    $p$-group.
\end{corollary}

We also observe that almost the same thing can be said about
nilpotent groups contained in $\cald$, namely:

\begin{theorem}
    Let $G$ be a finite nilpotent group contained in $\cald$. Then
    either $G$ is cyclic and $|G|\in\{p^n,p^n q^m\}$,
    where $p,q$ are primes and $n,m\in \mathbb{N}$,
    or $G$ is a non-cyclic $p$-group which does not contain
    abelian sections of type $(p,p,p)$.
\end{theorem}

Unfortunately, we were not able to determine all non-cyclic
$p$-groups with the above property. It is obvious that this is
satisfied by rank 2 abelian $p$-groups, but several non-abelian
$p$-groups also satisfy it, as shows the following theorem.

\begin{theorem}
    All finite $p$-groups having a cyclic maximal subgroup are
    contained in $\cald$. In particular, the non-abelian $p$-groups
    $M(p^n)$, $D_{2^n}$, $Q_{2^n}$ for $n\geq 3$, and $QD_{2^n}$
    for $n\geq 4$ are contained in $\cald$.
\end{theorem}

\begin{proof}
Let $n\geq 3$ be an integer and denore by $\calg_n$ the class of $p$-groups of order
$p^n$ possessing a cyclic maximal subgroup. This contain the cyclic $p$-group $\mathbb{Z}_{p^n}$,
the abelian $p$-group $\mathbb{Z}_p\times\mathbb{Z}_{p^{n-1}}$, as well as the non-abelian
$p$-groups described exhaustively in Theorem 4.1 of \cite{10}, II:
\begin{itemize}
\item[--] the modular $p$-group
$$M(p^n)=\langle x,y\mid x^{p^{n-1}}=y^p=1,\, y^{-1}xy=x^{p^{n-2}+1}\rangle,$$where $n\geq 4$ for $p=2$,
\item[--] the dihedral group
$$D_{2^n}=\langle x,y\mid x^{2^{n-1}}=y^2=1,\, yxy=x^{-1}\rangle,$$
\item[--] the generalized quaternion group
$$Q_{2^n}=\langle x,y\mid x^{2^{n-1}}=y^4=1,\, yxy^{-1}=x^{2^{n-1}-1}\rangle,$$
\item[--] the quasi-dihedral group
$$QD_{2^n}=\langle x,y\mid x^{2^{n-1}}=y^2=1,\, yxy=x^{2^{n-2}-1}\rangle,$$where $n\geq 4$.
\end{itemize}Notice that for every non-cyclic $p$-group
$G\in\calg_n$ we have $|G:\Phi(G)|=p^2$. Moreover, all maximal
subgroups of a group in $\calg_n$ belong to $\calg_{n-1}$.
\smallskip

We will prove by induction on $n$ that $\calg_n$ is included in
$\cald$. Clearly, this holds for $n=3$. Assume now that $n$ is a
minimal positive integer such that $\calg_n\nsubseteq\cald$ and
take $G\in\calg_n\setminus\cald$. Let
$\{H_1,K_1,H_2,K_2,...,H_m,K_m\}$ be a crown of order $2m$ in
$L(G)$. By our assumption, we infer that there is no maximal
subgroup $M$ of $G$ such that $K_i\subseteq M$ for all
$i=1,2,...,m$. Thus there exist at least two maximal subgroups
$M_1$ and $M_2$ of $G$ containing subgroups in the set
$\calk=\{K_1,K_2,...,K_m\}$. Since every $H_i$ is included in the
intersection of two subgroups in $\calk$, it follows that we can
choose $H_{i_1}\neq H_{i_2}$ which are contained in $M_1\cap
M_2=\Phi(G)$. In other words, the cyclic $p$-group $\Phi(G)$
contains two incomparable subgroups, a contradiction.
\end{proof}

An important class of finite nilpotent groups is constituted by
hamiltonian groups, that is non-abelian groups all of whose
subgroups are normal. The structure of such a group $H$ is
well-known, namely
$$H\cong Q_8 \times \mathbb{Z}_2^n \times A\,,$$where $Q_8$ is
the quaternion group, $n\in\mathbb{N}$ and $A$ is a finite abelian
group of odd order. By Theorem 2.5 we easily infer that if $H$
belongs to $\cald$, then $n=0$ and $A$ is trivial. In this way, a
nice characterization of $Q_8$ is obtained.

\begin{corollary}
    The quaternion $Q_8$ is the unique hamiltonian group contained in $\cald$.
\end{corollary}

In d) and e) of Theorem 1.3 two types of semidirect products of
orders $p q^{m}$ and $p^2q$, respectively, have been presented.
These belong to $\cald$ because their subgroup lattices are
planar. We end our note by studying the containment to $\cald$ for
other remarkable semidirect products: dihedral groups (notice that
these are not of type d) or e)). Recall that the dihedral group
$D_{2n}$ is the symmetry group of a regular polygon with $n$ sides
and it has the order $2n$. The most convenient abstract
description of $D_{2n}$ is obtained by using its generators: a
rotation $x$ of order $n$ and a reflection $y$ of order $2$. Under
these notations, we have
$$D_{2n}=\langle x,y\mid x^n=y^2=1,\ yxy=x^{-1}\rangle.$$The subgroup structure of $D_{2n}$ is
precisely known: for every divisor $d$ or $n$, $D_{2n}$ possesses
a subgroup isomorphic to $\mathbb{Z}_d$, namely $\langle
x^{\frac{n}{d}}\rangle$, and $\frac{n}{d}$ subgroups isomorphic to
$D_{2d}$, namely $\langle x^{\frac{n}{d}},x^{i-1}y\rangle$,
$i=1,2,...,\frac{n}{d}$\,.

\begin{theorem}
    The dihedral group $D_{2n}$ is contained in $\cald$ if and
    only if $n$ is of type $p^m$, where $p$ is a prime and
    $m\in\mathbb{N}$.
\end{theorem}

\begin{proof}
Assume first that $D_{2n}\in\cald$. By Lemma 2.1 we infer that
$n=2^kp^m$ with $p$ an odd prime and $k,m\in\mathbb{N}$. If $m=0$,
then $n$ is of the desired type. For $m\geq 1$ we will prove that
$k=0$. Suppose $k\geq 1$. Then it is easy to see that the
subgroups $H_1=\langle x^{2^{k-1}p^m}\rangle$, $K_1=\langle
x^{2^{k-1}p^m},y\rangle$, $H_2=\langle y\rangle$, $K_2=\langle
x^{2^k},y\rangle$, $H_3=\langle x^{2^k}\rangle$ and $K_3=\langle
x\rangle$ form a crown of order 6 in $L(D_{2n})$, a contradiction.
\smallskip

Conversely, we already know that $D_{2^{m+1}}$ belongs to $\cald$
by Theorem 2.6. Therefore we can assume that $p$ is odd. We will
prove by induction on $m$ that $D_{2p^m}$ is also contained in
$\cald$. This is obviously true for $m=1$ since $D_{2p}$ and
$\mathbb{Z}_p\times\mathbb{Z}_p$ are lattice-isomorphic. Let $m$
be minimal with the property $D_{2p^m}\notin\cald$ and take a
crown of order $2r$ in $L(D_{2p^m})$, say
$\{H_1,K_1,H_2,K_2,...,H_r,K_r\}$. We observe that the maximal
subgroups of $D_{2p^m}$ are $\langle
x\rangle\cong\mathbb{Z}_{p^m}$ and $\langle
x^p,x^{i-1}y\rangle\cong D_{2p^{m-1}}$, $i=1,2,...,p$. Since by
inductive hypothesis all $K_i$'\,s cannot be contained in the same
maximal subgroup of $D_{2p^m}$ and by using a similar argument as
in the proof of Theorem 2.6, one obtains that there are $H_i\neq
H_j$ included in $\Phi(D_{2p^m})$. But $\Phi(D_{2p^m})=\langle
x^p\rangle\cong\mathbb{Z}_{p^{m-1}}$ is cyclic, a contradiction.
Hence $D_{2p^m}\in\cald$.
\end{proof}

Finally, we indicate a natural open problem concerning the above
study.

\bigskip\noindent{\bf Open problem.} Characterize \textit{arbitrary} finite groups $G$
contained in $\cald$. Is it true that all these groups are
metacyclic?
\bigskip

\noindent{\bf Acknowledgments.} We are grateful to Professor
Roland Schmidt for his advices on the first version of our note.

\vspace*{5ex}\small

\hfill
\begin{minipage}[t]{5cm}
Marius T\u arn\u auceanu \\
Faculty of  Mathematics \\
``Al.I. Cuza'' University \\
Ia\c si, Romania \\
e-mail: {\tt tarnauc@uaic.ro}
\end{minipage}

\end{document}